\documentclass[letterpaper,12pt]{amsart}
\usepackage{amscd, amssymb}
\usepackage[driverfallback=hypertex]{hyperref}
\usepackage{comment}
\usepackage{enumitem}
\setlist[enumerate,1]{label={(\alph*)}}
\input xy
\xyoption{all}

\newtheorem{tm}{Theorem}[section]
\newtheorem{pr}[tm]{Proposition}
\newtheorem{lm}[tm]{Lemma}
\newtheorem{cy}[tm]{Corollary}

\theoremstyle{definition}

\theoremstyle{remark}
\newtheorem{rem}[tm]{Remark}

\newcommand{\R}{\mathbb R}
\newcommand{\C}{\mathbb C}

\renewcommand{\L}{\mathcal L}

\renewcommand{\H}{\mathcal H}

\DeclareMathOperator{\im}{Im}
\DeclareMathOperator{\re}{Re}
\DeclareMathOperator{\id}{Id}
\DeclareMathOperator{\vol}{vol}

\DeclareMathOperator{\ham}{Ham}

\DeclareMathOperator{\tr}{tr}

\DeclareMathOperator{\hess}{Hess}
\newcommand{\CC}{\mathcal C}
\newcommand{\OO}{\mathcal O}

\newcommand{\X}{\mathcal X}
\newcommand{\G}{\mathcal G}

\newcommand{\la}{\left\langle}
\newcommand{\ra}{\right\rangle}

\begin{document}
\title{Curvature of the space of positive Lagrangians}
\author{Jake P. Solomon}
\date{Oct. 2013}
\begin{abstract}
A Lagrangian submanifold in an almost Calabi-Yau manifold is called positive if the real part of the holomorphic volume form restricted to it is positive. An exact isotopy class of positive Lagrangian submanifolds admits a natural Riemannian metric. We compute the Riemann curvature of this metric and show all sectional curvatures are non-positive. The motivation for our calculation comes from mirror symmetry. Roughly speaking, an exact isotopy class of positive Lagrangians corresponds under mirror symmetry to the space of Hermitian metrics on a holomorphic vector bundle. The latter space is an infinite-dimensional analog of the non-compact symmetric space dual to the unitary group, and thus has non-positive curvature.
\end{abstract}

\maketitle

\pagestyle{plain}

\tableofcontents

\section{Introduction}
We aim to study the geometry of the space of Lagrangian submanifolds in a symplectic manifold. Specifically, let $(X,\omega)$ be a symplectic manifold of dimension $2n,$ and let $L$ be an oriented smooth manifold of dimension $n.$ Let $\L = \L(X,L)$ denote the space of oriented Lagrangian submanifolds $\Gamma \subset X$ diffeomorphic to $L.$ When $L$ is non-compact, we impose that all $\Gamma \in \L(X,L)$ coincide with a given $\Gamma_0$ outside a compact subset. Following Akveld and Salamon~\cite{AS01}, we regard $\L$ as an infinite dimensional manifold. Denote by $\Lambda : [0,1] \to \L$ a smooth path, and write $\Lambda(t) = \Lambda_t.$ The tangent space $T_\Gamma \L$ is canonically isomorphic to the space of compactly supported closed $1$-forms on $\Gamma.$ So, we identify the derivative $\frac{d}{dt} \Lambda_t$ with a closed $1$-form on $\Lambda_t.$ By definition, the path $\Lambda_t$ is \emph{exact} if the exists a family of functions $h_t : \Lambda_t \to \R$ such that
\[
\frac{d}{dt}\Lambda_t = dh_t.
\]
The path $\Lambda_t$ is \emph{compactly supported} if the functions $h_t$ can be chosen to have compact support.

We focus on the case that $X$ is an almost Calabi-Yau manifold. Namely, $X$ is equipped with a complex structure $J$ compatible with $\omega$ and a nowhere vanishing holomorphic $n$-form $\Omega.$ A Lagrangian submanifold $\Gamma \subset X$ is \emph{positive} if $\re \Omega|_\Gamma$ is a volume form. In the literature, positive Lagrangians are also called almost calibrated~\cite{CL04,Ne07}. Denote by $\L^+ = \L^+(X,L) \subset \L(X,L)$ the subspace of positive Lagrangian submanifolds. The space $\L^+$ is an open submanifold of $\L.$ Let $\OO \subset \L^+$ be a compactly supported exact isotopy class of positive Lagrangian submanifolds. That is, $\OO$ is the collection of all $\Gamma \in \L^+$ that can be connected to a fixed point in $\L^+$ by an exact compactly supported path. The isotopy class $\OO$ is a submanifold of $\L^+,$ and for $\Gamma \in \OO$ the tangent space $T_\Gamma\OO$ is canonically isomorphic to the space of exact $1$-forms on $\Gamma$ with compactly supported primitive. Let $\H_\Gamma$ denote the space of smooth $h : \Gamma \to \R$ satisfying the following normalization: If $\Gamma$ is compact then $\int_\Gamma h \re \Omega= 0$ and if $\Gamma$ is non-compact then $h$ has compact support. We identify $\H_\Gamma \simeq T_\Gamma \OO$ by $h \mapsto dh.$ Let
\[
(\cdot,\cdot) : \H_\Gamma \times \H_\Gamma \longrightarrow \R
\]
be given by
\[
(h,k) = \int_\Gamma hk \re \Omega.
\]
We think of $(\cdot,\cdot)$ as a Riemannian metric on the isotopy class $\OO.$ Let $\widetilde \OO$ denote the universal cover. The metric $(\cdot,\cdot)$ was previously studied in~\cite{So12}, where the author constructed a functional $\CC : \widetilde \OO \to \R$ convex with respect to $(\cdot,\cdot)$ and with critical points special Lagrangian submanifolds.

The main result of the present paper is a formula for the curvature of the metric $(\cdot,\cdot).$ To write the formula, we introduce the following notation. Let $g$ be the Riemannian metric on $X$ defined by
\[
g(\cdot,\cdot) = \omega(\cdot,J\cdot).
\]
Let $\rho : X \to \R_{>0}$ be the positive valued function defined by
\[
\rho^n\,\omega^n/n! = (-1)^{\frac{n(n-1)}{2}}(\sqrt{-1}/2)^n \Omega \wedge \overline \Omega.
\]
That is, $\rho$ measures the deviation of $g$ from being a Calabi-Yau metric. In particular, if $\rho = 1,$ then $g$ has vanishing Ricci curvature. For $\Gamma\in \L,$ let $\theta : \Gamma \to S^1$ be the function such that
\begin{equation}\label{eq:lagang}
\Omega|_{\Gamma} = e^{\sqrt{-1}\theta}\rho^{n/2} \vol.
\end{equation}
The existence of such $\theta$ follows from the equality case of the special Lagrangian inequality of Harvey-Lawson~\cite{HL82}. We denote by $\langle\cdot,\cdot\rangle$ the metric induced on $\Gamma$ by $g.$
\begin{tm}\label{tm:r4}
The Riemannian curvature of $(\cdot,\cdot)$ at $\Gamma \in \OO$ is determined by the following formula. For $h,k,l,m \in \H_\Gamma \simeq T_\Gamma \OO,$ we have
\[
(R(h,k)l,m) = -\int_\Gamma \sec \theta\left[ \la dh,dm\ra \la dk,dl\ra - \la dh,dl\ra \la dk,dm\ra  \right] \rho^{\frac{n}{2}} \vol.
\]
\end{tm}
Theorem~\ref{tm:r3} gives an explicit formula for $R(h,k)l,$ which is used to prove Theorem~\ref{tm:r4}. The formula of Theorem~\ref{tm:r4} is simpler than that of Theorem~\ref{tm:r3} because of several cancellations that result from integration by parts. Theorem~\ref{tm:lc} shows the existence of the Levi-Civita connection for $(\cdot,\cdot).$

The following corollary is immediate from Theorem~\ref{tm:r4}.
\begin{cy}\label{cy:sec}
The sectional curvature of $(\cdot,\cdot)$ is non-positive. More specifically, for $h,k \in \H_\Gamma \simeq T_\Gamma \OO,$ we have
\begin{equation}\label{eq:sec}
K(h,k) = -\frac{ \int_\Gamma \sec \theta\left[ \la dh,dh\ra \la dk,dk\ra - \la dh,dk\ra^2  \right] \rho^{\frac{n}{2}} \vol}{(h,h)(k,k) - (h,k)^2} \leq 0.
\end{equation}
\end{cy}
\begin{rem}\label{rem:flat}
For each $l \in \H_\Gamma,$ define $F_l \subset \H_\Gamma$ by
\[
F_l = \{ h \in \H_\Gamma| \exists c \in \R \text{ and } a \in C^\infty(\R) \text{ such that } h = a \circ l + c \}.
\]
Then $F_l$ is a flat subspace. That is, $K|_{F_l} = 0.$ Indeed, for any $h,k \in F_l,$ the differentials $dh$ and $dk$ are collinear at each point of $\Gamma,$ so the numerator of formula~\eqref{eq:sec} vanishes by the equality case of the Cauchy-Schwartz inequality.
\end{rem}

Riemannian metrics on infinite dimensional spaces and their curvature have received a good deal of attention. The curvature of Ebin's $L^2$ metric~\cite{Ebi70} on the space of Riemannian metrics was computed in~\cite{FrG89,GMM91} to be non-positive. A related pseudo-metric appeared in work of DeWitt~\cite{DeW67}. Ebin's metric restricted to the space of K\"ahler metrics on a complex manifold was studied by Calabi in~\cite{Cal54,Cal54a} and treated in detail by Calamai~\cite{Cal12}. Calabi's metric has positive sectional curvature. Clarke and Rubinstein~\cite{ClR13} analyze in depth the relationship between the metrics of Calabi and Ebin. A different metric on the space of K\"ahler metrics was studied by Mabuchi~\cite{Mab87}, Semmes~\cite{Sem92}, and Donaldson~\cite{Don99a}. Its curvature was shown to be non-positive. A related metric with non-positive curvature appears in Donaldson's work on Nahm's equations~\cite{Don10}. The Riemannian structure on the space of Hermitian metrics on a holomorphic vector bundle played an important role in Donaldson's work on the Kobayashi-Hitchin correspondence~\cite{Do85,Do87}. An expression for its curvature is given below.

A novel aspect of the present work is the somewhat mysterious nature of the space $\OO$ under consideration. The works just cited deal with spaces of metrics or functions satisfying a positivity condition, which being convex subsets of a linear space, have trivial topology. With the exception of~\cite{Don10}, the curvature computations are to one extent or another motivated by possibly formal symmetric space structures. On the other hand, a priori, the topology of the space $\OO$ could be extremely complicated. The author is not aware of a formal symmetric space structure for $\OO.$ Instead, the motivation for computing the curvature comes from mirror symmetry.

Briefly, homological mirror symmetry~\cite{Ko95} predicts an equivalence between the derived Fukaya category of a Calabi-Yau manifold $X$ and the derived category of coherent sheaves on a mirror Calabi-Yau manifold $X^\vee.$ In particular, Lagrangian submanifolds of $X$ should be roughly analogous to holomorphic vector bundles over $X^\vee.$ Based on evidence from~\cite[Section 5.5]{So12}, it seems reasonable that an exact isotopy class $\OO \subset \L^+(X,L)$ is analogous under mirror symmetry to the space of Hermitian metrics on a holomorphic vector bundle over $X^\vee.$

Recall the Riemannian structure of the space of Hermitian metrics on a holomorphic vector bundle. Namely, let $E \to X^\vee$ be a holomorphic vector bundle and let $Herm^+(E)$ be the space of Hermitian metrics on $E.$ Denote the K\"ahler form of $X^\vee$ by $\omega$. The space $Herm^+(E)$ is an open subset of the linear space of not necessarily positive definite Hermitian structures $Herm(E).$ Thus, for $H \in Herm^+(E)$ there is a canonical isomorphism $T_H Herm^+(E) \simeq Herm(E).$
Using this isomorphism, for two tangent vectors $\xi,\eta\in T_H Herm(E),$ define an inner product by
\[
(\xi,\eta) = \int_{X^\vee} \tr(H^{-1}\xi H^{-1}\eta) \omega^n.
\]
We think of $(\cdot,\cdot)$ as a Riemannian metric on $Herm^+(E).$ See~\cite{Ko87} for an extended discussion. The space $Herm^+(E)$ equipped with the metric~$(\cdot,\cdot)$ is an infinite dimensional analog of the symmetric space $Herm^+(n)$ of Hermitian inner products on a complex vector space of dimension $n.$ In particular,
\[
Herm^+(n) = GL_\C(n)/U(n), \qquad Herm^+(E) = GL_\C(E)/U(E),
\]
where $GL_\C(E)$ is the group of all complex-linear automorphisms of $E,$ and $U(E)$ is the group of all automorphisms of $E$ that are unitary for a fixed Hermitian metric. The formula for the curvature of $Herm^+(n)$ is standard~\cite{PeP06}. The same computation yields the following formula for the curvature of $Herm^+(E).$ For $\xi,\eta,\zeta,\lambda \in T_H Herm^+(E),$ we have\footnote{For an explanation of the factor $\frac{1}{4},$ see the derivation of formula (1.12) of~\cite{FrG89}.}
\[
(R(\xi,\eta)\zeta,\lambda) = -\frac{1}{4}\int_{X^\vee} \tr([H^{-1}\xi,H^{-1}\eta][H^{-1}\lambda,H^{-1}\zeta]) \omega^n.
\]
Thus, for the sectional curvature, we obtain
\[
K(\xi,\eta) = -\frac{1}{4}\frac{\int_{X^\vee} \tr([H^{-1}\xi,H^{-1}\eta]^2) \omega^n}{(\xi,\xi)(\zeta,\zeta) - (\xi,\zeta)^2} \leq 0.
\]
Flat subspaces of $T_H Herm^+(E)$ are given by families of $\xi \in Herm^+(E)$ such that the corresponding endomorphisms $H^{-1}\xi$ are simultaneously diagonalizable. Note the analogy between the forgoing discussion and Theorem~\ref{tm:r4}, Corollary~\ref{cy:sec} and Remark~\ref{rem:flat}.

The analogy is not perfect. For example, generalizing from the finite dimensional case, the curvature tensor of $Herm^+(E)$ can be shown to be covariant constant. On the other hand, the curvature tensor of $\OO$ appears not to be covariant constant. Indeed, we have taken considerable freedom in our interpretation of homological mirror symmetry in terms of submanifolds and vector bundles. Thus, the author found it surprising that the analogy holds as far as it does.

The results of the present article along with the results of~\cite{So12} suggest the following program of research. It is desirable to develop a satisfactory existence theory for geodesics, or approximations thereto, for the space $\OO.$ The existence theory for geodesics should imply that the Riemannian metric $(\cdot,\cdot)$ induces a metric space structure on $\OO.$ By an argument analogous to~\cite{CaC02}, Theorem~\ref{tm:r4} should then imply that $\OO$ has non-positive curvature in the sense of Alexandrov. This raises the prospect of proving long time existence for minimizing movements of the convex functional $\CC$ of~\cite{So12} by analogy with the work of~\cite{StJ12} for Mabuchi's $K$-energy functional. The gradient flow of $\CC$ is closely related to Lagrangian mean curvature flow, the long-time existence of which has been the subject of intense research~\cite{TY02,CL04,Ne07,NeA13}. Ultimately, long-time existence of the gradient flow for $\CC$ should lead to existence criteria for special Lagrangian submanifolds. Existence theory for geodesics along with the convexity of $\CC$ would also imply uniqueness for special Lagrangian submanifolds in $\OO,$ thus providing an alternative to the method of~\cite{TY02}. A similar argument was used previously to prove uniqueness for Einstein-Hermitian metrics on holomorphic vector bundles~\cite{Do85} and for extremal K\"ahler metrics~\cite{ChT08}. The author plans to pursue this program in future work.

The paper is organized as follows. Section~\ref{sec:lnt} proves refined versions of Weinstein's Lagrangian neighborhood theorem in Propositions~\ref{pr:weiref} and~\ref{pr:orth}. Section~\ref{sec:fle} reviews the essentials of the space $\L$ and uses Proposition~\ref{pr:orth} to construct a convenient local coordinate system on $\OO$ around any $\Gamma \in \OO.$ Section~\ref{sec:lc} constructs the Levi-Civita connection for the metric $(\cdot,\cdot)$ and establishes related formulas. Section~\ref{sec:pre} proves several lemmas that play an important role in the curvature calculation. Section~\ref{sec:main} combines the results of the preceding sections to prove a formula for $R(h,k)l$ in Theorem~\ref{tm:r3} and deduce from it Theorem~\ref{tm:r4}.

\subsection{Acknowledgements}
The author is grateful to S.~Donaldson, Y.~Rubinstein, J.~Streets, and G.~Tian for helpful conversations. During the preparation of this manuscript, the author was partially supported by Israel Science Foundation grant 1321/2009 and Marie Curie grant No. 239381.

\section{Lagrangian neighborhood theorem}~\label{sec:lnt}
In Section~\ref{sec:coo}, we construct a convenient coordinate system on $\OO$ centered at an arbitrary $\Gamma \in \OO.$ The main tool in the construction is a refinement of Weinstein's Lagrangian neighborhood theorem, which we discuss here.

Let $(X,\omega)$ be a symplectic manifold and let $\Gamma \subset X$ be a Lagrangian submanifold. Let $\omega_0$ denote the canonical symplectic form on the cotangent bundle $T^*\Gamma.$ We think of $\Gamma$ also as the submanifold of $T^*\Gamma$ given by the image of the zero section. Weinstein~\cite{Wei71} proved the following Lagrangian neighborhood theorem. We follow the exposition of~\cite{MS98}.
\begin{tm}\label{tm:wei}
There exist a neighborhood $U \subset T^*\Gamma$ of $\Gamma$, a neighborhood $V\subset X$ of $\Gamma$ and a diffeomorphism $\psi : U \to V$ such that
\[
\psi^*\omega = \omega_0, \qquad \psi|_\Gamma = \id.
\]
\end{tm}
The diffeomorphism $\psi$ of Theorem~\ref{tm:wei} is not uniquely determined, and we need to take advantage of its flexibility. In the following, keep in mind that the differential of the inclusion of $\Gamma$ in $T^*\Gamma$ as the zero section makes $T\Gamma$ a subbundle of $TT^*\Gamma|_\Gamma$ in a canonical way. Similarly, the differential of the inclusion of $\Gamma$ in $X$ makes $T\Gamma$ a subbundle of $TX|_\Gamma$ in a natural way.
\begin{pr}\label{pr:weiref}
Let $A : TT^*\Gamma|_\Gamma \to TX|_\Gamma$ be an isomorphism of vector bundles such that $\omega(A\cdot,A\cdot) = \omega_0(\cdot,\cdot)$ and $A|_{T\Gamma} = \id.$
The diffeomorphism $\psi$ of Theorem~\ref{tm:wei} can be chosen to satisfy the additional condition
\[
d\psi|_\Gamma = A.
\]
\end{pr}
The proof of Proposition~\ref{pr:weiref} depends on the following lemma.
\begin{lm}\label{lm:ham}
Let $B$ be an automorphism of the vector bundle $TT^*\Gamma|_{\Gamma}$ preserving $\omega_0$ such that $B|_{T\Gamma} = \id.$ Then there exists a symplectomorphism $\chi$ of $T^*\Gamma$ such that
\[
\chi|_\Gamma = \id, \qquad d\chi|_\Gamma = B.
\]
\end{lm}
\begin{proof}
The canonical inclusion of $T^*\Gamma$ in $TT^*\Gamma|_\Gamma$ as the vertical tangent space, together with the inclusion of $T\Gamma$ in $TT^*\Gamma$ from the differential of the zero section induce a direct sum decomposition
\[
TT^*\Gamma|_\Gamma \simeq T\Gamma \oplus T^*\Gamma.
\]
With respect to this direct sum decomposition, write $B$ as a block matrix
\[
B =
\begin{pmatrix}
P & Q \\
R & S
\end{pmatrix},
\]
where
\begin{gather*}
P \in End(T\Gamma), \qquad S \in End(T^*\Gamma), \\
Q \in Hom(T^*\Gamma,T\Gamma), \qquad R \in Hom(T\Gamma,T^*\Gamma).
\end{gather*}
By assumption $P = \id$ and $R = 0.$ Since $B$ preserves $\omega_0,$ it follows that $S = \id$ and $Q$ is self-dual. That is, if we think of $Q$ as a section of $T\Gamma \otimes T\Gamma,$ it is a symmetric tensor. So, $Q$ defines a symmetric bilinear form on $T^*\Gamma.$  Define the Hamiltonian function $H : T^*\Gamma \to \R$ by, for $p \in \Gamma$ and $\alpha \in T_p^*\Gamma,$
\[
H(p,\alpha) = \frac{1}{2}Q_p(\alpha,\alpha).
\]
Let $\xi_H$ be the symplectic gradient of $H,$
\[
i_{\xi_H}\omega_0 = dH,
\]
and let $\chi_t$ be the associated Hamiltonian flow,
\[
\chi_0 = \id, \qquad \frac{d\chi_t}{dt} = \xi_H \circ \chi_t.
\]
We claim that $\chi = \chi_1$ satisfies the conditions of the lemma. Indeed, since $dH$ vanishes along $\Gamma,$ it follows that $\xi_H|_\Gamma = 0.$ Therefore $\chi|_\Gamma = \id.$ Moreover, the time derivative of $d\chi_t$ satisfies
\[
d\chi_0 = \id, \qquad \frac{d}{dt} d\chi_t|_\Gamma = \nabla \xi_H \circ d\chi_t|_\Gamma.
\]
Here, the derivatives $\frac{d}{dt} d\chi_t|_\Gamma$ and $\nabla \xi_H|_\Gamma$ are well-defined without choosing a connection because $\chi_t |_\Gamma = \id$ and $\xi_H|_\Gamma = 0.$ It is easy to see that
\[
\nabla \xi_H|_\Gamma  =
\begin{pmatrix}
0 & Q \\
0 & 0
\end{pmatrix} \in End(TT^*\Gamma|_\Gamma).
\]
Therefore,
\[
d\chi_1 = \exp(\nabla \xi_H) =
\begin{pmatrix}
\id & Q \\
0 & \id
\end{pmatrix}
=B.
\]
\end{proof}
\begin{proof}[Proof of Proposition~\ref{pr:weiref}]
Let $\psi$ be the diffeomorphism of Theorem~\ref{tm:wei}. Then
\[
B = (d\psi|_\Gamma)^{-1}\circ A
\]
satisfies the conditions of Lemma~\ref{lm:ham}. Let $\chi$ be the symplectomorphism given by Lemma~\ref{lm:ham}. Then
\[
d(\psi \circ \chi)|_\Gamma = d\psi|_\Gamma \circ B = A.
\]
So, replacing $\psi$ with $\psi \circ \chi,$ and $U$ with $\chi^{-1}(U),$ we obtain Proposition~\ref{pr:weiref}.
\end{proof}

In the following proposition, we take $(X,\omega,J)$ to be a K\"ahler manifold. That is, $J$ is an integrable complex structure on $X$ compatible with $\omega.$ Let $g$ be the K\"ahler metric, $g(\cdot,\cdot) = \omega(\cdot,J\cdot).$
\begin{pr}\label{pr:orth}
The diffeomorphism $\psi$ from Theorem~\ref{tm:wei} can be chosen so that, with respect to the metric $\psi^*g,$ the fibers of $\pi:T^*\Gamma \to \Gamma$ are perpendicular to $\Gamma.$ That is, for
\[
x \in \Gamma, \qquad \xi \in T_x (\pi^{-1}(x)), \qquad \eta \in T_x \Gamma,
\]
we have $\psi^*g(\xi,\eta) = 0.$
\end{pr}
\begin{proof}
As in the proof of Lemma~\ref{lm:ham}, we consider the canonical splitting into horizontal and vertical tangent spaces, $TT^*\Gamma|_\Gamma \simeq T\Gamma \oplus T^*\Gamma.$
Let $A: TT^*\Gamma|_\Gamma \to TX|_\Gamma$ be the unique vector bundle isomorphism such that
\[
\omega(A\cdot,A\cdot) = \omega_0(\cdot,\cdot), \qquad A|_{T\Gamma} = \id, \qquad A(T^*\Gamma) = J(T\Gamma) \subset TX|_\Gamma.
\]
Proposition~\ref{pr:weiref} with this choice of $A$ gives $\psi$ satisfying the desired condition.
\end{proof}

\section{Families of Lagrangian embeddings}\label{sec:fle}
\subsection{Conventions}
Let $\Theta$ be a manifold. Throughout the paper, we say a family $\{\phi_{\theta}\}_{\theta \in \Theta}$ of maps of manifolds $X\to Y$ is \emph{smooth} if there exists a smooth map
\[
\phi : \Theta \times X \to Y
\]
such that $\phi_\theta = \phi|_{\{\theta\} \times X}.$

We denote the space of differential $k$-forms on a manifold $X$ by $A^k(X)$ and the compactly supported ones by $A_c^k(X).$
Let $f : X \to Y$ be a map of smooth manifolds, let $v$ be a vector field along $f$ and let $\rho \in A^k(Y).$ Let $\xi_1,\ldots,\xi_{k-1},$ be vector fields on $X.$ We extend the interior product to a map
\[
i_v : A^k(Y) \to A^{k-1}(X)
\]
by the formula
\[
(i_v \rho)(\xi_1,\ldots,\xi_{k-1})|_x = \rho_{f(x)}(v(x),df_x(\xi_1(x)),\ldots,df_x(\xi_{k-1}(x))).
\]

\subsection{The space of Lagrangians}
We summarize the discussion of Akveld-Salamon~\cite[Section 2]{AS01}. Let $\X = \X(X,L)$ denote the space of Lagrangian embeddings
\[
f : L \to X, \qquad f^* \omega = 0.
\]
If $L$ is non-compact, we impose that all $f \in \X$ agree with a given $f_0$ outside a compact subset of $L.$ Let $\G$ denote the group of compactly supported orientation preserving diffeomorphisms of $L.$ Define an action
\[
\G \times \X \to \X
\]
by
\[
(\psi,f) \mapsto f \circ \psi.
\]
The space of Lagrangian submanifolds $\L(X,L)$ is defined to be the quotient $\L = \X/\G.$ The equivalence class $\Lambda=[f]$ of a Lagrangian embedding $f : L \to X$ is identified with the submanifold $f(L)\subset X.$

Let $\Theta$ be a manifold. Let $\Lambda : \Theta \to \L$ be a map. Write $\Lambda_\theta = \Lambda(\theta).$ We say that $\Lambda$ is \emph{smooth} if there exists a smooth lifting $f$ such that the following diagram commutes:
\[
\xymatrix{&\X\ar[d] \\
\Theta \ar[ur]^{f} \ar[r]^{\Lambda} & \L
}
\]
Let $\Lambda: (-\epsilon,\epsilon) \to \L$ be a smooth path with lifting $f : (-\epsilon,\epsilon)\to \X.$ Let $v_t$ be the vector field along $f_t$ defined by
\[
v_t = \frac{d}{dt} f_t.
\]
Write
\[
\alpha_t = (f_t)_*i_{v_t}\omega \in A^1(\Lambda_t).
\]
Then
\[
d\alpha_t = (f_t)_*\frac{d}{dt} f_t^* \omega = 0.
\]
The following lemma is due to Akveld and Salamon~\cite[Lemma 2.1]{AS01}.
\begin{lm}\label{lm:tan}
The $1$-form $\alpha_t$ is independent of the choice of lifting $f_t$ of $\Lambda_t.$ So, for $\Gamma \in \L$ there is a canonical isomorphism
\[
T_\Gamma \L \overset{\sim}{\longrightarrow} \{\rho \in A^1_c(\Gamma)|d\rho = 0\}
\]
sending the equivalence class of a smooth path $\Lambda : (-\epsilon,\epsilon)\to \L$ with $\Lambda_0 = \Gamma$ to $\alpha_0 \in A^1(\Gamma).$
\end{lm}
We denote by $Ham(X,\omega)$ the compactly supported Hamiltonian symplectomorphism group of $X.$ The following lemma is the same as~\cite[Lemma 2.2]{AS01}.
\begin{lm}\label{lm:un}
Let $\Lambda : [a,b] \to \L$ be a smooth path. Let $\phi : [a,b] \to \ham(X,\omega)$ be a smooth path with $\phi_0 = \id_X,$ generated by the time dependent Hamiltonian $H_t.$ Then $\phi_t(\Lambda_0) = \Lambda_t$ if and only if
\[
\frac{d}{dt} \Lambda_t = dH_t|_{\Lambda_t}
\]
for all $t.$
\end{lm}

\subsection{Coordinates}\label{sec:coo}
In this section, we construct certain families of Lagrangian embeddings and establish related notation that will be used throughout the paper. These families provide convenient coordinate charts on the exact isotopy class $\OO.$

Let $(X,\omega,J)$ be a K\"ahler manifold. Let $\Gamma \subset X,$ and let $f_0 : L \to \Gamma$ be a diffeomorphism. Remembering the inclusion of $\Gamma$ in $X,$ we also think of $f_0$ as a Lagrangian embedding.
Let
\[
U \subset T^*\Gamma, \qquad V \subset X, \qquad \psi: U \longrightarrow V,
\]
be as in Theorem~\ref{tm:wei}. We identify $U$ and $V$ using $\psi$ and write $g$ for $\psi^*g$ and $J$ for $\psi^*J$ accordingly. Let $\pi$ denote the restriction to $U$ of the projection $T^*\Gamma \to \Gamma.$ By Proposition~\ref{pr:orth}, we assume the fibers of $\pi$ are $g$-perpendicular to the zero section. In short, we define a coordinate chart from a neighborhood of zero in $\H_\Gamma$ to a neighborhood of $\Gamma$ in $\OO$ by assigning
\[
h \mapsto graph(dh) \subset U \subset T^*\Gamma.
\]

We proceed to develop notation to facilitate calculations in this coordinate system. Given a finite family of functions $h^i\in \H_\Gamma,$ we construct a family of Lagrangian embeddings $L \to X$ near $f_0$ as follows. Let
\begin{equation}\label{eq:H}
H^i = h^i \circ \pi : U \longrightarrow \R.
\end{equation}
Let $\xi^i$ denote the symplectic gradient of $H^i,$
\[
i_{\xi^i} \omega_0 = dH^i.
\]
Since the fibers of $\pi$ are Lagrangian submanifolds, the vector fields $\xi^i$ are tangent to the fibers of $\pi.$ Thus
\begin{equation}
d\pi(\xi^i) = 0, \label{eq:dpx}
\end{equation}
and it follows that
\begin{equation}\label{eq:comm}
\{H^i,H^j\} = 0, \qquad [\xi^i,\xi^j] = 0.
\end{equation}
Let $t_i \in \R$ be a family of real parameters, and write $t = (t_i).$ Let $\phi_t$ denote the simultaneous Hamiltonian flow of the Poisson commuting functions $H^i.$ That is, $\phi_t$ is defined by
\[
\phi_0 = \id, \qquad
\partial_{t_i} \phi_t = \xi^i \circ \phi_t.
\]
Since the functions $h^i$ have compact support, $\phi_t$ exists for $t$ sufficiently small. We implicitly assume hereafter that $t$ is sufficiently small. It follows from equation~\eqref{eq:dpx} that for all $t,$
\begin{equation}\label{eq:piphi}
\pi \circ \phi_t = \pi.
\end{equation}
Let
\begin{equation}\label{eq:ft}
f_t = \phi_t \circ f_0 : L \to X, \qquad \Gamma_t = \im f_t.
\end{equation}
In other words, the Lagrangian submanifold $\Gamma_t$ is the graph of $\sum_i t_i dh^i$ in $T^*\Gamma$ and $f_t = (\pi|_{\Gamma_t})^{-1}\circ f_0.$ Let
\[
h^i_t = H^i|_{\Gamma_t}.
\]
When it does not cause confusion, we may write $h^i = h_t^i.$
By Lemma~\ref{lm:un}, we have
\begin{equation}\label{eq:Hi}
\partial_{t_i}\Gamma_t = dH^i|_{\Gamma_t} = dh^i_t.
\end{equation}
We denote by $u^i = u^i_t$ the vector field along $f_t$ given by
\[
u^i = \partial_{t_i} f_t.
\]
By the definition~\eqref{eq:ft} of $f_t,$ we have
\begin{equation}\label{eq:uxi}
u^i = \xi^i \circ f_t.
\end{equation}
We frequently use the fact that
\begin{equation}\label{eq:JH}
\xi^i = -J\nabla H^i.
\end{equation}
Our assumption that the fibers of $\pi$ are perpendicular to the zero section implies that $\nabla H^i|_\Gamma$ is tangent to $\Gamma.$ In particular,
\begin{equation}\label{eq:Hh}
\nabla H^i|_\Gamma = \nabla h^i.
\end{equation}
Here $\nabla H^i$ denotes the gradient of $H^i$ with respect to $g$ and $\nabla h^i$ denotes the gradient of $h^i$ with respect to the induced metric $\la\cdot,\cdot\ra$ on $\Gamma.$

\section{Levi-Civita connection}\label{sec:lc}
We now prove the existence of the Levi-Civita connection of the Riemannian metric $(\cdot,\cdot)$ on $\OO$ and establish relevant notation. In the present section, $(X,\omega,J,\Omega)$ denotes an almost Calabi-Yau manifold, and $\OO$ is a compactly supported exact isotopy class in $\L^+(X,L).$

Let $\Lambda : [0,1] \to \OO$ and write $\Lambda_t = \Lambda(t).$ Let $h_t: \Lambda_t \to \R$ denote a family of functions. Let $g_t : L \to \Lambda_t \subset X$ be a family of parametrizations, and let $u_t$ be the vector field along $g_t$ given by
\[
u_t = \frac{d}{dt} g_t.
\]
Write
\[
\widetilde \Omega = g_t^* \Omega, \qquad \tilde h_t = h_t \circ g_t.
\]
Let $w_t$ be the vector field on $L$ such that
\[
i_{w_t}\re\widetilde\Omega = -i_{u_t}\re \Omega.
\]
We define the covariant derivative $D$ by
\[
\frac{D}{dt} h_t = \left(\frac{d}{dt} \tilde h_t  + w_t \cdot \tilde h_t\right) \circ g_t^{-1}.
\]
Intuitively, the choice of parametrizations $g_t$ gives a framing \[
C^\infty_c(L)\times[0,1] \longrightarrow \Lambda^*T\OO
\]
by $(h,t) \mapsto h \circ g_t^{-1}.$ The differential operator $w_t$ is the Christoffel symbol determined by this framing. It is not hard to see that $\frac{D}{dt}h_t|_{t = 0}$ depends only on $\frac{d}{dt}\Lambda_t|_{t = 0}$ and that the dependence is linear.
The reader may wish to consult~\cite[Section 5]{So12} to see the author's original motivation for this definition. The main result of the present section is the following.
\begin{tm}\label{tm:lc}
The connection $D$ is the Levi-Civita connection of $(\cdot,\cdot).$
\end{tm}
Before proving the theorem, we introduce the notation necessary to discuss covariant derivatives in the context of the families of Lagrangian embeddings introduced in Section~\ref{sec:coo}. This will be used both in proving Theorem~\ref{tm:lc} and in curvature calculations. Recall that the Lagrangian angle $\theta$ is defined by equation~\eqref{eq:lagang}. Using the notation of Section~\ref{sec:coo}, we write
\[
\widetilde\Omega = \widetilde \Omega_t = f_t^* \Omega, \qquad \tilde \theta =\tilde\theta_t = \theta \circ f_t, \qquad \tilde h^i = H^i \circ f_t.
\]
One of the key advantages of the family $f_t$ is that $\tilde h^i$ is independent of $t.$ Indeed, since $\{H^i,H^j\} = 0,$ we have
\[
\tilde h^i = H_i \circ f_0 = h_i \circ f_0.
\]
It follows from equation~\eqref{eq:Hi} that
\[
f_t^*\partial_{t_i}\Gamma_t = d\tilde h^i.
\]
Let $w^i = w^i_t$ be the vector field on $L$ defined by
\begin{equation}\label{eq:w}
i_{w^i_t} \re \widetilde \Omega_t = -i_{u^i_t} \re \Omega.
\end{equation}
By equations~\eqref{eq:uxi} and~\eqref{eq:JH}, we have
\begin{equation}\label{eq:uH}
i_{u^i_t} \re \Omega = f^*_t i_{\nabla H^i} \im \Omega.
\end{equation}
Thus by equations~\eqref{eq:lagang} and~\eqref{eq:Hh}, we have
\begin{equation}\label{eq:wi0}
w^i_0 = -\tan \tilde\theta \, \nabla \tilde h^i.
\end{equation}
The following lemmas will be used in the proof of Theorem~\ref{tm:lc} as well as later in the paper.
\begin{lm}\label{lm:deg}
Let $\alpha,\beta,$ be differential forms on a manifold $M$ with $\deg \alpha + \deg \beta > \dim M,$ and let $\xi$ be a vector field on $M.$ Then
\[
(i_\xi \alpha) \wedge \beta = (-1)^{\deg \alpha + 1} \alpha \wedge i_\xi \beta.
\]
\end{lm}
\begin{proof}
By the degree assumption $\alpha \wedge \beta = 0.$ The lemma follows from the derivation property of the interior product.
\end{proof}
\begin{lm}\label{lm:Djk}
We have
\[
D_{h^j} h^k \circ f_t = w^j \cdot \tilde h^k = -\frac{d \tilde h^k \wedge i_{u^j} \re \Omega}{\re \widetilde\Omega} = -\frac{ d\tilde h^k \wedge f_t^* i_{\nabla H^j} \im \Omega}{\re \widetilde \Omega}.
\]
\end{lm}
\begin{proof}
The first equality is essentially the definition of $D$ together with the fact that $\tilde h^i$ is time independent. The third equality follows immediately from equation~\eqref{eq:uH}. It remains to prove the second equality. Indeed, by Lemma~\ref{lm:deg} with $\alpha = d\tilde h^k,\,\beta = \re \widetilde \Omega$ and $\xi = w^j,$ we obtain
\[
w^j \cdot \tilde h^k \, \re \widetilde \Omega = d\tilde h^k \wedge i_{w_j} \re \widetilde \Omega.
\]
The claim now follows by defining equation~\eqref{eq:w}.
\end{proof}
\begin{proof}[Proof of Theorem~\ref{tm:lc}]
First we prove that $D$ is a metric connection. Namely,
\begin{equation*}
\partial_{t_i} (h^j,h^k) - (D_{h^i} h^j,h^k) - (h^j,D_{h^i} h^k) = 0.
\end{equation*}
Indeed,
\[
(h^j,h^k) = \int_{\Gamma_t}h^j h^k \re \Omega = \int_L \tilde h^j \tilde h^k \re \widetilde \Omega.
\]
So, since $\tilde h^j$ and $\tilde h^k$ are independent of $t,$ by equation~\eqref{eq:w} we obtain
\begin{align*}
&\partial_{t_i} (h^j,h^k) = \int_L \tilde h^j \tilde h^k \partial_{t_i} \re \widetilde\Omega = \int_L \tilde h^j \tilde h^k di_{u^i} \re \Omega \\
& \qquad \qquad \qquad = -\int_L \tilde h^j \tilde h^k di_{w^i} \re\widetilde\Omega = -\int_L \tilde h^j \tilde h^k \L_{w^i} \re \widetilde\Omega.
\end{align*}
Thus by the first equality of Lemma~\ref{lm:Djk}, the Leibniz rule and Stokes' theorem, we have
\begin{align*}
&\partial_{t_i} (h^j,h^k) - (D_{h^i} h^j,h^k) - (h^j,D_{h^i} h^k) = \\
& \qquad = -\int_L \tilde h^j \tilde h^k \L_{w^i} \re \widetilde\Omega - \int_L (w^i \cdot \tilde h^j) \tilde h^k \re \widetilde \Omega - \int_L \tilde h^j (w^i \cdot \tilde h^k)  \re \widetilde \Omega \\
& \qquad = -\int_L \L_{w^i} (\tilde h^j \tilde h^k \re \widetilde \Omega) = 0.
\end{align*}

It remains to prove that $D$ is symmetric. That is,
\[
D_{h^j} h^k = D_{h^k} h^j.
\]
We begin by observing that since $\omega$ is of type $(1,1)$ and $\Omega$ is of type $(n,0),$ we have
\[
\omega \wedge \re \Omega = 0.
\]
Applying $i_{\xi^j}i_{\xi^k}$ to the preceding equation, and using the fact that $\omega(\xi^k,\xi^j) = \{H^j,H^k\} = 0,$ we obtain
\[
dH^j \wedge i_{\xi^k}\Omega - dH^k \wedge i_{\xi^j}\Omega + \omega \wedge i_{\xi^j} i_{\xi^k} \Omega = 0.
\]
Pulling back by $f_t^*$ and using the fact that $f_t^*\omega = 0,$ we deduce that
\[
d\tilde h^j \wedge i_{u^k} \re \Omega - d\tilde h^k \wedge i_{u^j} \re \Omega = 0.
\]
Therefore, by Lemma~\ref{lm:Djk} we have
\[
(D_{h^j} h^k - D_{h^k} h^j)\circ f_t = -\frac{d \tilde h^k \wedge i_{u^j} \re \Omega}{\re \widetilde\Omega} + \frac{d \tilde h^j \wedge i_{u^k} \re \Omega}{\re \widetilde\Omega} = 0.
\]
\end{proof}
\begin{rem}
A more conceptual explanation of why $D$ is metric can be found in~\cite{So12}. The argument showing that $D$ is symmetric is essentially a special case of \cite[Lemma 3.1]{So12} combined with Lemma~\ref{lm:Djk}.
\end{rem}

\section{Preliminaries}\label{sec:pre}
\subsection{A Lie bracket}\label{ssec:bra}
We continue with the combined notation and assumptions of Section~\ref{sec:coo} and Section~\ref{sec:lc}. Let $g$ denote the K\"ahler metric on $X$, and let $\overline\nabla$ denote the Levi-Civita connection of $g.$ Let
\[
N_\Gamma = TM_\Gamma^\perp \subset TX|_\Gamma
\]
denote the normal bundle of $\Gamma.$ Let
\[
A : T\Gamma \otimes T\Gamma \longrightarrow N_\Gamma
\]
denote the second fundamental form. That is, for $\xi,\eta,$ vector fields on $\Gamma,$
\[
A(\xi,\eta) = (\overline\nabla_\xi \hat \eta)^\perp,
\]
where $\hat \eta$ is any extension of $\eta$ to $X.$ It is well known that
\[
A(\xi,\eta) = A(\eta,\xi).
\]
We denote by $\nabla$ the Levi-Civita connection of $\Gamma = \Gamma_0$ as well as the gradient, both with respect to the induced metric $\la\cdot,\cdot\ra.$ The following lemmas will be important in the curvature calculation.
\begin{lm}\label{lm:bra}
We have
\[
[\xi^i,\nabla H^j]|_{\Gamma} = \overline\nabla_{\xi_i}\nabla H^j|_{\Gamma} + JA(\nabla h^j,\nabla h^i) + J \nabla_{\nabla h^j} \nabla h^i.
\]
\end{lm}
\begin{proof}[Proof of Lemma~\ref{lm:bra}]
Using the fact that $\overline\nabla J = 0$ along with equation~\eqref{eq:JH}, we have
\begin{align*}
[\xi^i,\nabla H^j]|_{\Gamma} &= \left(\overline \nabla_{\xi^i}\nabla H^j - \overline\nabla_{\nabla H^j} \xi^i \right) |_\Gamma\\
& = \left(\overline \nabla_{\xi^i}\nabla H^j + J \overline \nabla_{\nabla H^j}\nabla H^i\right)|_{\Gamma}.
\end{align*}
By equation~\eqref{eq:Hh}, we have
\begin{align*}
\overline \nabla_{\nabla H^j}\nabla H^i|_{\Gamma} &= \left(\overline \nabla_{\nabla H^j}\nabla H^i\right)^\perp + \left(\overline \nabla_{\nabla H^j}\nabla H^i\right)^\parallel \\
&=  A(\nabla h^i,\nabla h^j) +\nabla_{\nabla h^j}\nabla h^i,
\end{align*}
which implies the lemma.
\end{proof}
\begin{lm}\label{lm:symm}
We have the symmetry
\[
\overline\nabla_{\xi^i} \nabla H^j = \overline \nabla_{\xi^j} \nabla H^i.
\]
\end{lm}
\begin{proof}
Using the fact that $\overline \nabla J = 0$ along with equations~\eqref{eq:comm} and~\eqref{eq:JH}, we calculate
\begin{equation*}
0 = [\xi^i,\xi^j] = \overline \nabla_{\xi^i} \xi^j - \overline \nabla_{\xi^j}\xi^i = -J \left( \overline\nabla_{\xi^i}\nabla H^j - \overline\nabla_{\xi^j} \nabla H^i\right).
\end{equation*}
The lemma follows.
\end{proof}

\subsection{Derivative of the Lagrangian angle}\label{ssec:dlagang}
In the following, we generalize part of~\cite[Lemma 2.3]{TY02} to the almost Calabi-Yau setting. Let $\Lambda : [0,1] \to \OO$ and write $\Lambda_t = \Lambda(t).$ Let $h: \Lambda_0 \to \R$ be a function such that $\frac{d}{dt}\Lambda_t|_{t = 0} = dh.$ Let $g_t : L \to \Lambda_t \subset X$ be a family of parametrizations such that $\frac{d}{dt} g_t|_{t=0} = -J\nabla h \circ g_t.$ Let $\theta = \theta_t: \Lambda_t \to S^1$ be the Lagrangian angle of $\Lambda_t.$ Let $\triangle$ denote the Laplace-Beltrami operator on functions on $\Lambda_0$ with respect to the induced metric.
\begin{lm}\label{lm:dth}
We have
\[
\left.\frac{d}{dt}\theta\circ g_t\right|_{t =0} = \left(\triangle h - \frac{n}{2\rho} \la d\rho,dh\ra\right)\circ g_0.
\]
\end{lm}
\begin{proof}
Taking the derivative of the pull-back of equation~\eqref{eq:lagang}, we obtain
\[
\frac{d}{dt} g_t^*\Omega  = \sqrt{-1}\frac{d}{dt}(\theta \circ g_t) \,g_t^*(e^{\sqrt{-1}\theta}\rho^{n/2}\vol) + e^{\sqrt{-1}\theta\circ g_t} \frac{d}{dt} g_t^* (\rho^{n/2}\vol).
\]
On the other hand, by Cartan's formula and equation~\eqref{eq:lagang},
\begin{align*}
\left.\frac{d}{dt}g^*_t\Omega\right|_{t = 0}&= g_0^*di_{-J\nabla h} \Omega= -\sqrt{-1}g_0^*d i_{\nabla h} \Omega =\\
 &= g_0^*\left[\left(e^{\sqrt{-1}\theta}\rho^{n/2}d\theta \wedge *dh \right) - \vphantom{\frac{n}{2\rho}}\right.\\
&\qquad\qquad\left.-\sqrt{-1}e^{\sqrt{-1}\theta}\rho^{n/2}\left(\frac{n}{2\rho} d\rho \wedge *dh - \triangle h \vol \right)\right] \\
& = g_0^*\left\{\left(e^{\sqrt{-1}\theta}\rho^{n/2} \vol \right) \left[ \la d\theta, dh \ra  - \vphantom{\frac{n}{2\rho}}\right.\right.\\
&\qquad\qquad\qquad\qquad\qquad\left.\left.-\sqrt{-1}\left(\frac{n}{2\rho} \la d\rho,dh \ra - \triangle h  \right)\right]\right\}.
\end{align*}
Combining the two equations, dividing by $g_0^*(e^{\sqrt{-1}\theta}\rho^{n/2}\vol)$ and extracting the imaginary part, we obtain the desired result.
\end{proof}

\section{The main computation}\label{sec:main}
The following theorem will be shown to imply Theorem~\ref{tm:r4}.
\begin{tm}\label{tm:r3}
For $h,k,l \in \H_\Gamma \simeq T_\Gamma\OO,$ we have
\begin{align*}
&R(h,k)l= \\
&\quad = -\sec^2\theta\left[\left(\triangle h - \frac{n}{2\rho} \la dh,d\rho\ra\right)\la dk,dl\ra - \right. \\
& \quad\qquad \qquad \qquad \qquad\qquad \qquad-\left.\left(\triangle k - \frac{n}{2\rho}\la dk,d\rho\ra\right) \la dh,dl\ra\right]+ \\
& \quad\qquad\qquad + \sec^2\theta\left(\la\nabla_{\nabla h}\nabla k,\nabla l\ra -  \la \nabla_{\nabla k}\nabla h,\nabla l\ra \right)  + \\
& \quad\qquad \qquad + \tan \theta \sec^2\!\theta \left(\la \nabla h,\nabla \theta\ra \la \nabla k,\nabla l\ra -\la \nabla k,\nabla\theta\ra \la \nabla h,\nabla l\ra\right).
\end{align*}
\end{tm}
To prove Theorem~\ref{tm:r3}, we make use of the identification between $\H_\Gamma$ and a neighborhood of $\Gamma$ given in Section~\ref{sec:coo}. Thus, functions $h^i,h^j,h^k \in \H_\Gamma,$ are identified with coordinate vector fields. The main ingredient in the proof is the following proposition.
\begin{pr}\label{pr:Dijk}
We have
\begin{align*}
D_{h^i}D_{h^j}h^k|_{t = 0} &= - \sec^2 \theta \la dh^k,dh^j\ra\left(\triangle h^i - \frac{n}{2\rho}\la d\rho,dh^i\ra\right) \\
& \qquad -\frac{ d h^k \wedge i_{\overline \nabla_{\xi^i}\nabla H^j + JA(\nabla h^j,\nabla h^i)}\im \Omega }{\re \Omega|_{\Gamma}}\\
& \qquad-\la \nabla h^k,\nabla_{\nabla h^j}\nabla h^i \ra \\
& \qquad + \tan \theta \sec^2 \theta \la \nabla h^i,\nabla \theta\ra \la \nabla h^j,\nabla h^k\ra \\
&\qquad + \tan^2\theta \left(\la \nabla_{\nabla h^i} \nabla h^j,\nabla h^k\ra + \la \nabla h^j,\nabla_{\nabla h^i}\nabla h^k\ra \right).
\end{align*}
\end{pr}
\begin{proof}
In the notation of Section~\ref{sec:lc}, by the definition of $D,$
\begin{equation}\label{eq:Dijk}
D_{h^i}D_{h^j}h^k \circ f_0 = \left.\partial_{t_i}\left( D_{h^j}h^k\circ f_t\right)\right|_{t = 0} + w^i \cdot \left(D_{h^j}h^k \circ f_0 \right).
\end{equation}
Using Lemma~\ref{lm:Djk} and equation~\eqref{eq:Hh}, we calculate
\begin{align}\label{eq:dtDjk}
\partial_{t_i}\left( D_{h^j}h^k \circ f_t\right)|_{t = 0} &= -\left.\partial_{t_i}\frac{ d\tilde h^k \wedge f_t^* i_{\nabla H^j} \im \Omega}{\re \widetilde \Omega}\right|_{t = 0} \\
&= -\frac{ d\tilde h^k \wedge f_0^* i_{[\xi^i,\nabla H^j]} \im \Omega}{\re \widetilde \Omega} - \notag\\
& \qquad - \frac{ d\tilde h^k \wedge i_{\nabla \tilde h^j} \partial_{t_i}\im \widetilde\Omega|_{t = 0}}{\re \widetilde \Omega}+\notag\\
& \qquad + \frac{ d\tilde h^k \wedge i_{\nabla \tilde h^j} \im \widetilde\Omega}{\re\widetilde \Omega}\times\frac{\partial_{t_i}\re\widetilde \Omega|_{t = 0}}{\re \widetilde\Omega}.\notag
\end{align}
By Lemma~\ref{lm:deg} with $\xi = \nabla \tilde h^j,\,\alpha = d\tilde h^k,$ and $\beta = \partial_{t_i}\im \widetilde\Omega,\,\im\widetilde\Omega,$ we obtain
\begin{align}\label{eq:dtI}
&- \frac{ d\tilde h^k \wedge i_{\nabla \tilde h^j} \partial_{t_i}\im \widetilde\Omega|_{t = 0}}{\re \widetilde \Omega}
+ \frac{ d\tilde h^k \wedge i_{\nabla \tilde h^j} \im \widetilde\Omega}{\re\widetilde \Omega}\times\frac{\partial_{t_i}\re\widetilde \Omega|_{t = 0}}{\re \widetilde\Omega} = \\
& \qquad  = - \frac{ \la d\tilde h^k,d\tilde h^j\ra \partial_{t_i}\im \widetilde\Omega|_{t = 0}}{\re \widetilde \Omega}
+ \frac{\la d\tilde h^k,d\tilde h^j\ra \im \widetilde\Omega}{\re\widetilde \Omega}\times\frac{\partial_{t_i}\re\widetilde \Omega|_{t = 0}}{\re \widetilde\Omega} \notag\\
& \qquad = -\la d\tilde h^k,d\tilde h^j\ra \left.\partial_{t_i} \frac{\im \widetilde\Omega}{\re\widetilde\Omega}\right|_{t = 0} \notag\\
& \qquad = -\la d\tilde h^k,d\tilde h^j\ra \partial_{t_i}\tan (\theta \circ f_t)|_{t = 0}.\notag
\end{align}
By Lemma~\ref{lm:dth}, we have
\begin{equation}\label{eq:dttan}
\partial_{t_i}\tan (\theta \circ f_t)|_{t = 0} = \left[\sec^2\theta \left(\triangle h^i - \frac{n}{2\rho}\la dh^i,d\rho \ra\right)\right] \circ f_0.
\end{equation}
Lemma~\ref{lm:bra} gives
\begin{align}\label{eq:dtII}
&-\frac{ d\tilde h^k \wedge f_0^* i_{[\xi^i,\nabla H^j]} \im \Omega}{\re \widetilde \Omega} = \\
& \; = -\frac{ d\tilde h^k \wedge f_0^*\left(i_{\overline \nabla_{\xi^i}\nabla H^j + JA(\nabla h^j,\nabla h^i)}\im \Omega + i_{\nabla_{\nabla h^j} \nabla h^i} \re \Omega\right)}{\re \widetilde \Omega}\notag\\
& \; = -\left[\frac{ d h^k \wedge i_{\overline \nabla_{\xi^i}\nabla H^j + JA(\nabla h^j,\nabla h^i)}\im \Omega }{\re \Omega|_{\Gamma}}+\la \nabla h^k,\nabla_{\nabla h^j}\nabla h^i \ra \right]\circ f_0.\notag
\end{align}
Combining equations~\eqref{eq:dtDjk},~\eqref{eq:dtI},~\eqref{eq:dttan} and~\eqref{eq:dtII}, we obtain
\begin{align}\label{eq:DI}
&\partial_{t_i}\left( D_{h^j}h^k \circ f_t\right)|_{t = 0} = \\
& \;  = \left[-\sec^2\theta \la dh^j,dh^k\ra\left(\triangle h^i - \frac{n}{2\rho}\la dh^i,d\rho\ra \right) - \right.\notag\\
&\quad \left. \vphantom{\frac{n}{2\rho}} -\frac{ d h^k \wedge i_{\overline \nabla_{\xi^i}\nabla H^j + JA(\nabla h^j,\nabla h^i)}\im \Omega }{\re \Omega|_{\Gamma}}-\la \nabla h^k,\nabla_{\nabla h^j}\nabla h^i \ra\right] \circ f_0. \notag
\end{align}
By Lemma~\ref{lm:Djk} and equation~\eqref{eq:wi0},
\begin{align}\label{eq:DII}
&w^i \cdot \left(D_{h^j}h^k \circ f_0 \right) = \tan\tilde\theta \,\nabla \tilde h^i \cdot \left(\tan\tilde\theta \la d\tilde h^j,d\tilde h^k\ra\right)  = \\
& \qquad\quad = \tan \tilde\theta \sec^2 \tilde\theta \la \nabla \tilde h^i,\nabla \tilde\theta\ra \la \nabla \tilde h^j,\nabla \tilde h^k\ra+ \notag\\
& \qquad\qquad \quad +\tan^2 \tilde\theta \left(\la \nabla_{\nabla \tilde h^i} \nabla \tilde h^j,\nabla \tilde h^k\ra + \la \nabla \tilde h^j,\nabla_{\nabla \tilde h^i}\nabla \tilde h^k\ra \right).\notag
\end{align}
Adding equations~\eqref{eq:DI} and~\eqref{eq:DII} gives the proposition.
\end{proof}
\begin{proof}[Proof of Theorem~\ref{tm:r3}]
To obtain the theorem, we anti-symmetrize the formula of Proposition~\ref{pr:Dijk} in $i$ and $j.$ The term on the second line of the formula is symmetric in $i$ and $j$ by Lemma~\ref{lm:symm} and the symmetry of the second fundamental form. The term $\tan^2\theta \la \nabla h^j,\nabla_{\nabla h^i} \nabla h^k\ra$ on the last line of the formula is symmetric in $i$ and $j$ by the symmetry of the Hessian. So, these terms do not contribute. Further, the following terms combine:
\begin{align*}
&-\la \nabla h^k, \nabla_{\nabla h^j} \nabla h^i \ra + \la \nabla h^k, \nabla_{\nabla h^i} \nabla h^j\ra + \\
&\qquad \qquad +\tan^2 \theta \left(\la \nabla h^k, \nabla_{\nabla h^i} \nabla h^j\ra-\la \nabla h^k, \nabla_{\nabla h^j} \nabla h^i \ra\right) = \\
& \qquad \qquad \qquad \qquad = \sec^2 \theta \left(\la \nabla h^k, \nabla_{\nabla h^i} \nabla h^j\ra-\la \nabla h^k, \nabla_{\nabla h^j} \nabla h^i \ra\right).
\end{align*}
The theorem follows.
\end{proof}

\begin{proof}[Proof of Theorem~\ref{tm:r4}]
By Theorem~\ref{tm:r3} and equation~\eqref{eq:lagang}, we have
\begin{align*}
&(R(h,k)l,m) = \int_\Gamma (R(h,k)l)m \cos \theta \rho^{\frac{n}{2}}\vol = \\
&= -\int_\Gamma \left [\sec \theta \left(\triangle h \la dk,dl\ra - \triangle k \la dh, dl\ra\right)+ \vphantom{\frac{n}{2\rho}}\right. \\
& \qquad + \frac{n}{2\rho}\sec \theta \left( \la dh,d\rho\ra\la dk,dl\ra - \la dk, d\rho\ra \la dh,dl\ra\right) \\
& \qquad + \sec\theta\left(\la\nabla_{\nabla h}\nabla k,\nabla l\ra -  \la \nabla_{\nabla k}\nabla h,\nabla l\ra \right)  + \\
& \qquad\left.\vphantom{\frac{n}{2\rho}} + \tan \theta \sec\theta \left(\la \nabla h,\nabla \theta\ra \la \nabla k,\nabla l\ra -\la \nabla k,\nabla\theta\ra \la \nabla h,\nabla l\ra\right)\right]
m\rho^{\frac{n}{2}} \vol.
\end{align*}
Integrating the Laplacians by parts cancels the second, third and fourth lines, when the gradient falls on $\rho^{\frac{n}{2}},$ the inner product, and $\sec \theta$ respectively. The remaining part of the gradient of the inner product cancels due to the symmetry of $\hess l.$ The contribution from when the gradient falls on $m$ is the desired formula.
\end{proof}

\bibliographystyle{../amsabbrvc}
\bibliography{../ref}

\vspace{.5 cm}
\noindent
Institute of Mathematics \\
Hebrew University, Givat Ram \\
Jerusalem, 91904, Israel \\

\end{document}